\documentclass[12pt]{amsart}
\usepackage{amsmath, amssymb, amscd, a4wide, mathrsfs, mathptmx}

\usepackage{relsize}
\usepackage[colorlinks=true,linkcolor=blue,citecolor=blue,pdfpagelabels=false]{hyperref}
\numberwithin{equation}{section}

\theoremstyle{plain}
\newtheorem{thm}{Theorem}[section]
\newtheorem{lem}[thm]{Lemma}
\newtheorem{prop}[thm]{Proposition}
\newtheorem{cor}[thm]{Corollary}
\newcommand{\thmref}[1]{Theorem~\ref{#1}}
\newcommand{\lemref}[1]{Lemma~\ref{#1}}
\newcommand{\propref}[1]{Proposition~\ref{#1}}
\newcommand{\corref}[1]{Corollary~\ref{#1}}

\theoremstyle{definition}
\newtheorem{rmk}[thm]{Remark}

\newtheorem{conj}{Conjecture}

\newtheorem*{hypo*}{Hypothesis~I}

\newcommand{\hyporef}[1]{Hypothesis~I}

\newcommand{\mf}{\mathbf}
\newcommand{\q}{\quad}

\newcommand{\mc}{\mathcal}

\newcommand{\mrm}{\mathrm}

\begin{document}

\title[Sign change, Yoshida lifts]{The first negative eigenvalue of Yoshida lifts}

\author{Soumya Das}
\address{Department of Mathematics\\
	Indian Institute of Science\\
	Bangalore - 560012, India.}
\email{soumya@iisc.ac.in}

\author{Ritwik Pal}
\address{Department of Mathematics\\
	Indian Institute of Science\\
	Bangalore - 560012, India.}
\email{ritwikpal@iisc.ac.in}

\date{}

\begin{abstract}
We prove that for any given $\epsilon >0$, the first negative eigenvalue of the Yoshida lift $F$ of a pair of elliptic cusp forms $f,g$ having square-free levels (where $g$ has weight $2$ and satisfies $(\log Q_{g})^2 \ll \log Q_f$), occurs before $c_{\epsilon} \cdot Q_F^{1/2-2 \theta+ \epsilon} $; where $Q_F,Q_f,Q_g$ are the analytic conductors of $F,f,g$ respectively, $\theta < 1/4$, and $c_{\epsilon}$ is a constant depending only on $\epsilon$. 
\end{abstract}
\maketitle 

\section{Introduction} 
Eigenvalues of Hecke eigenforms are of considerable interest to number theorists, in particular their distribution (e.g., with respect to the Sato-Tate measure), magnitude (Ramanujan-Petersson conjecture), and more recently study of their signs have been the focus of intensive research. In this paper we are interested about the signs of eigenvalues of the so-called Yoshida lifts, whose definition would be recalled below. Let us first briefly discuss the setting of the problem and the results existing in the literature. The best known result for the first sign change of an elliptic newform was given by K. Matom\"aki \cite{km1} : if $Q_{f}$ is the analytic conductor (see section $2$ for the definition) of an elliptic newform $f$ of level $N$ and weight $k \geq 2$ (so $Q_f \asymp k^2N$, i.e., has the same order of magnitude as $k^2N$), then the first negative eigenvalue $a_n$ occurs for some $n \ll Q_{f}^{3/8}$. See also \cite{choie} for related results. 

There are far fewer results available in the context of Siegel modular forms. For a Siegel Hecke eigenform $F$ on $\mrm{Sp}_2(\mf Z)$, which is not a Maa{\ss} lift (so that $k \geq 20$), it is known \cite{KSieg} that its eigenvalues change signs infinitely many often. Related results are available in \cite{DK}, \cite{pita}. Concerning the first negative eigenvalue problem, the best known result \cite{siegel} due to Kohnen and Sengupta says that the first negative eigenvalue $\lambda_F(n)$ (with $F$ as above) occurs for 
\begin{equation} \label{1bd}
 n \ll Q_{F} \log^{20} Q_{F}, 
\end{equation}
the implied constant being absolute. Here $Q_F$ denotes the analytic conductor of $F$, defined in section~\ref{prelim}. This result was generalised to the case of higher levels (which were held fixed throughout the paper, but both Maa{\ss} lifts and non-lifts were considered) by J. Brown \cite{brown}, who got the same bound as above. Note that in both of these cases one has $Q_F \asymp k^2$. Improving these results seem to be a rather difficult problem. One of the main reasons behind this is that the Hecke relations between the eigenvalues of a Siegel-Hecke eigenform are more complicated than those for an elliptic newform. In this paper we restrict our attention to the case of the Yoshida lift of two elliptic cusp forms and show that one can improve the above results considerably (cf. \eqref{1bd}, \eqref{1/2bd}). 

The setting of this paper is as follows (see e.g. \cite{saha} for a more detailed discussion). Let $S_\kappa(\mc L)$ denote the space of cusp forms of even weight $\kappa \geq 2$ and level $\mc L$. Let $f \in S_{k}(N_1)$ and $g \in S_{2}(N_{2})$ be normalised newforms with $N_{1} , N_{2} \geq 1$ squarefree and $M:= \mrm{gcd} (N_{1}, N_{2}) > 1$. Assume that the Atkin-Lehner eigenvalues of $f$ and $g$ coincide for all $p$ dividing $M$. To this data, one can associate a Siegel modular form $F=F_{f,g} \in S_{k/2 +1} (\Gamma_{0}^{2} (N))$, where $N= \mrm{lcm} (N_{1}, N_{2})$, which is called the Yoshida lift attached to $f,g$ (see \cite{saha}, \cite{yoshida}). Let the Hecke eigenvalues of $F$ be $\lambda_{F} (n)$. Further, let $\theta$ denote any saving over the exponent of convexity (so that $\theta<1/4$, see \cite{subconvexity}) bound for the normalised $L$-functions $L(f,s),L(g,s)$ on the critical line. We prove the following theorem.     

\begin{thm} \label{mthm}
Let $\epsilon >0$ be given and the notation and setting be as in the above paragraph.  Suppose $(\log Q_{g})^2 \ll \log Q_f$. Then there exists $n \in \mathbf{N}$ with
\begin{equation} \label{1/2bd}
n \ll_{\epsilon} Q_{F}^{1/2-  2\theta+ \epsilon}  
\text{ such that } \lambda_{F}(n) <0 . 
\end{equation} 
\end{thm}

This result suggests that for a generic Siegel Hecke eigenform (not necessarily a lift) of degree $2$, the bound $Q_F^{1/2+\epsilon}$ could be plausible (since the Yoshida lift satisfies the Ramanujan-Petersson conjecture, which seems crucial in these problems), this is got by taking $\theta=0$ in the above theorem. Perhaps the stronger exponent $1/2-\delta$ could be true.

\begin{conj}
For an arbitrary Siegel Hecke eigenform $F$ of degree $2$ and weight $k \geq 20$, for any given $\epsilon >0$, the first negative eigenvalue $\lambda_F(n)$ occurs at $n \ll_{\epsilon} Q_{F}^{1/2+\epsilon} $, with the implied constant being absolute and depending only on $\epsilon$.
\end{conj}

Our proof of \thmref{mthm} uses the factorisation of the spinor $L$-function of $F$ as a product of two $\mrm{GL}_2$ $L$-functions, subconvexity estimates of the $\mrm{GL}_2$ $L$-functions in question, and the Hecke relations for the eigenvalues of the elliptic newforms $f,g$. The use of the subconvexity bound is not crucial for us, \thmref{mthm} with $\theta=0$ is already an improvement of Kohnen and Senupta's bound in \cite{siegel}. Like all other results on this topic, we consider upper and lower bounds for a suitable weighted sum of the eigenvalues of $F$:
\begin{align}
S(F,x) := {\sum_{n \leq x, \, (n,N)=1}} \lambda_{F}(n) \log (\frac{x}{n}) \label{sum1}
\end{align}
in terms of $Q_F$ and $x$. Using standard analytic techniques, we obtain an upper bound $Q_{F}^{1/4- \theta+\epsilon} \cdot x^{1/2}$ with an implied constant depending only on $\epsilon$. The main point is to get a suitable lower bound by exploiting the non-negativity of $\lambda_F(n)$, say up to $x$, and exploiting the Hecke relations. The two bounds combined would give the desired upper bound $Q_{F}^{1/2-2\theta+\epsilon}$ in \thmref{mthm}. Let us mention here that our method of exploiting the Hecke relations between eigenvalues is rather different from those existing in the literature for any other `Linnik-type' problem on determining the first sign change in the sequence of Hecke eigenvalues of an eigenform. See the paragraph below, and for more details, section~\ref{lbd}. 

For all $y$ such that $\log y \gg (\log \mc{L} )^2$ and any given elliptic newform  $h \in S_2(\mc L)$ ($2$ can be replaced by any fixed weight $\ge 2$), we prove a non-trivial upper bound of $\sum_{p \leq y, p \nmid \mathcal{L}}  |\lambda_{h} (p) |$   (one trivial bound is $2 \pi(y)$ with $\pi(\cdot)$ being the prime counting function, and we show that one can reduce the constant here to $11/10$); up to the best of our knowledge, this result has not been written down explicitly in the literature. This follows from the holomorphy of the symmetric power $L$-functions (cf. \cite{kimsha}) and can be used to provide point-wise upper bounds for $\lambda_h(p)$ on sets of primes with positive natural density (in an effective and explicit manner, in particular the set of primes may depend mildly on $\mc L$, see \corref{cor}). In particular we avoid using the Sato-Tate theorem not only because such an advanced machinery is not required (and so perhaps this method may generalise to other situations), but more so because we need explicit and controllable dependence on the parameter $\mc L$.

\subsection*{Acknowledements} We thank the referees for meticulous checking of the paper and for numerous comments and suggestions that improved the presentation. S.D. acknowledges financial support in parts from the UGC Centre for Advanced Studies, DST (India) and IISc. Bangalore during the completion of this work. R.P. thanks NBHM for the financial support and IISc., where this work was done. 

\section{Notation and preliminaries}\label{prelim}
\subsection{General notation}
Let $\mc A$ be a subset of $\mathbf{N}$ and $a_{n} \in \mathbf{C}$, we define 
\[  \underset{n \in \mc A}{\sum} {}^N a_n:=  \underset{ n \in \mc A, \, (n,N)=1}{\sum} a_n   . \]
Let $ u(x),v(x)$ be two real functions defined on a subset $\mc B$ of $\mathbf{R}$. Whenever we write $u(x) \ll v(x)$ or $u(x) = O(v(x))$ or $u(x) \ll_\epsilon v(x)$, it will always mean that,
\[ |u(x)| \leq M \cdot h(x), \quad \text{for all  } x \in \mc B \text{  and for some  } M >0 ,\]
and in the last case $M$ may depend on $\epsilon$. The notation $u(x)= o(v(x))$ means that 
\[ \underset{x \rightarrow \infty} \lim \frac{u(x)}{v(x)} =0 . \]

\subsection{Spinor \texorpdfstring{$L$}{}-function}
Whenever $H(s)$ is a Dirichlet series with an Euler product, having Euler factors $H_{p}(s)$ at primes $p$, we will write for $\Re s \gg 1$, $\mc L \geq 1$ and primes $p$  
\[ H_{\mathcal{L}}(s) := \underset{p \nmid \mathcal{L}} \prod H_{p}(s) \quad \text{and} \quad H^{\mathcal{L}} (s) :=  \underset{p | \mathcal{L}} \prod H_{p}(s);
\]
so that $H(s)= H_{\mathcal{L}}(s) \cdot H^{\mathcal{L}}(s)$.
Let the notation be as in the introduction. We can attach to the Yoshida lift $F \in S_{\kappa} (\Gamma_{0}^{2} (N))$ the spinor $L$-function $Z(F,s)$ (in the sense of Langlands) which is given by a certain Euler-product defined in terms of the Satake parameters of $F$. In this paper, we will always work with Euler-products away from $N$.

Useful information about the Euler factors of $Z(F,s)$ away from the level is given, for instance, in (see \cite[Prop.~6.1]{bo-pilot}). Let us consider the Euler factor $Z_{F,p}(s)$ of $Z(F,s)$ at a prime $p \nmid N$. For $\Re s>1$ we have that
\[   Z_{N}(F,s):= \underset{p \nmid N}\prod Z_{F,p}(s), \quad \text{where} \quad Z_{F,p}(s):= \underset{1\leq i \leq 4}\prod (1 - \beta_{i,p} p^{-s})^{-1}. \] 
Here $\beta_{1,p} := \alpha_{0,p}$, $\beta_{2,p} := \alpha_{0,p} \alpha_{1,p}$, $\beta_{3,p} := \alpha_{0,p} \alpha_{2,p}$ , $\beta_{4,p} := \alpha_{0,p} \alpha_{1,p} \alpha_{2,p}$, and the complex numbers $\alpha_{0,p},\alpha_{1,p},\alpha_{2,p}$ are the Satake parameters of $F$ at $p$ (see \cite{brown}). We normalise $Z(F,s)$ by substituting $s$ by $s+k-3/2$ in $Z(F,s)$. Through this choice of normalisation we have $\alpha_{0,p}^2 \alpha_{1,p} \alpha_{2,p}=1$.

The main information that we require about a Yoshida-lift is the following. Firstly, for $(n,N)=1$, $Z_{N}(F,s)$ is related to the eigenvalues $\lambda_F(n)$ of the Hecke operator $T(n)$ acting on $S_{\kappa} (\Gamma_{0}^{2} (N))$ by the following relation  (see \cite[Prop.~4.4]{brown})
\begin{equation} \label{naive-Spinor}
\underset{n \in \mathbf{N}} \sum {}^N \frac{\lambda_{F}(n)}{n^{s}} = \frac{Z_{N}(F,s)}{\zeta_{N}(1+2s)},
\end{equation}
where $ \zeta_{N}(s) := \underset{p \nmid N} \prod (1-p^{-s})^{-1}$ for $\Re s >1$. Note that the choice of normalisation of $Z(F,s)$ does not affect the sign of $\lambda_{F}(n)$ for any $n \in \mf N$. 

Secondly, letting $\lambda_{f}(n), \lambda_{g}(n)$ denote the normalised Hecke eigenvalues and $L(f,s), L(g,s)$ denote the normalised $L$-functions of $f$ and $g$, so that their functional equations relate $s$ with $1-s$:
\[  L(f,s) := \sum_{n=1}^\infty \lambda_f(n) n^{-s}, \q   L(g,s) := \sum_{n=1}^\infty \lambda_g(n) n^{-s} ; \]
we have the relation (see \cite[Prop.~3.1]{saha})
\begin{equation} \label{l-fn}
Z_{N} (F, s) = L_{N_1}(f,s) \cdot L_{N_2}(g,s) = \frac{L(f,s)}{L^{N_1}(f,s)} \cdot \frac{L(g,s)}{ L^{N_2}(g,s)},
\end{equation}
where $L^{N_1}(f,s)$ and $L^{N_2}(g,s)$ are given by
\[L^{N_1}(f,s) :=\underset{p | N_{1}} \prod (1- \lambda_{f}(p) p^{-s})^{-1} \quad \text{and} \quad  L^{N_2}(g,s) :=\underset{p | N_{2}} \prod (1- \lambda_{g}(p) p^{-s})^{-1}. \]

\subsection{Analytic conductor} Let $L(h,s)=  \sum_{n \geq 1} \lambda_{h}(n) n^{-s}$ be a normalised `$L$-function' of a modular form $h$ in the sense of \cite[Chap.~5]{iwaniec}. Further assume that $L(h,s)$ has an Euler product of degree $d$ with the $p$-Euler polynomial given by $\prod_{1 \leq i \leq d}(1-\alpha_i(p)X)$ for some complex numbers $\alpha_i(p)$ for all $p$. The completed $L$-function
\[
\Lambda(h,s)= q(h)^{\frac{s}{2}} \pi^{\frac{-ds}{2}} \prod_{j=1}^{d} \Gamma(\frac{s+k_{j}}{2}) L(h,s)
\]
satisfies a functional equation relating $s$ with $1-s$, has meromorphic continuation to $\mf C$, where $q(h) \geq 1$ is the arithmetic conductor and  $k_{j} \in \mathbf{C}$. When $p \nmid q(h)$, one has $\alpha_i(p)<p$. Then we define the analytic conductor (see \cite{iwaniec} for a more detailed discussion) $Q_{h}$ of $L(h,s)$ (or $h$ for brevity) as 
\[ Q_{h} := q(h) \prod_{j=1}^{d} (|k_j|+3).\]

For $f \in S_{k}(N_{1})$ and $g \in S_{2}(N_2)$ it is well-known that $Q_{f} \asymp k^2 N_{1}$ and $Q_{g} \asymp N_{2}$. Letting $Q_{F}$ denote the analytic conductor of $F$, from \eqref{l-fn} we have $Q_F=Q_f\cdot Q_g$. So
\begin{equation} \label{analytic conductor}
Q_{F} \asymp k^{2} N_{1} N_{2}.
\end{equation}

For $x \geq 1$, we put
\begin{equation} \label{sfx}
S(F,x) := \underset{n \leq x} \sum {}^N \lambda_{F}(n) \log (\frac{x}{n}) .
\end{equation}

\section{Upper and lower bounds for \texorpdfstring{$S(F,x)$}{}}
Let $\lambda_{F}(n) \geq 0$ for all $n \leq x$. We will estimate $x$ by comparing the upper and lower bounds of $S(F,x)$. We will first work with $f$, $g$ and finally transfer everything to $F$ using \eqref{l-fn}.

\subsection{Upper bound}
Let $Q_{h}$ be the analytic conductor of an elliptic Hecke newform $h$. From the subconvexity bound for $\mrm{GL}_{2}$ $L$-functions (see \cite[Theorem 1.1]{subconvexity}) we have for any $t \in \mf R$,
\begin{equation} \label{subconvexity}
| L(h, \frac{1}{2} +it) | \ll Q_{h}^{1/4-\theta} | \frac{1}{2}+it |^{1/2-2\theta} ,
\end{equation}
for some $1/4> \theta >0$.
From Perron's formula, \eqref{naive-Spinor} and \eqref{l-fn} we can write
\begin{equation} \label{Perron}
\underset{n \leq x} \sum {}^N \lambda_{F}(n) \log(\frac{x}{n}) = \frac{1}{2 \pi i } \int_{(2)} \frac{1}{\zeta_{N}(1+2s)} L_{N_1}(f,s)  L_{N_2}(g,s) \frac{x^{s}}{s^2} ds.
\end{equation}
For $p | N_1$, we have $ |\lambda_{f} (p)| \leq 1$ (see Theorem 3, \cite{li}). Also recall that $N_1$ is squarefree. Let us now note the following estimate that will be used in the next paragraph.
\begin{equation} \label{G}
|L^{N_1}(f, \frac{1}{2}+it)|^{-1} = \underset{p | N_{1}} \prod |1- \lambda_{f}(p) p^{-\frac{1}{2}-it}| \leq \underset{p | N_{1}} \prod (1+ \frac{|\lambda_{f}(p)|}{p^{\frac{1}{2}}}) \leq \underset{d | N_{1}} \sum \frac{1}{d^{\frac{1}{2}}} = O_{\epsilon}( N_{1}^{\epsilon}).
\end{equation}
Similarly we get, $|L^{N_2}(g, \frac{1}{2} + it)|^{-1}= O_{\epsilon}(N_2^{\epsilon})$.

From \eqref{l-fn}, \eqref{subconvexity}, \eqref{G} and the fact that $\zeta_{N}(2+2it)^{-1}$ is absolutely bounded, it is clear that the integral in \eqref{Perron} is absolutely convergent on the critical line $ \Re s= \frac{1}{2} $. Shifting the line of integration to $\Re (s)= \frac{1}{2}$ gives (the bound in \eqref{subconvexity} implies that the horizontal integrals do not contribute)
\[  S(F,x) = \frac{1}{2 \pi i }  \int_{(\frac{1}{2})}
\frac{1}{\zeta_{N}(1+2s)} L_{N_1}(f,s) L_{N_2}(g,s) \frac{x^{s}}{s^2} ds.
\]
Now again using \eqref{l-fn}, \eqref{analytic conductor}, \eqref{subconvexity}, \eqref{G} we estimate in a standard way that
\begin{equation} \label{upbd}
S(F,x) \ll_{\epsilon} Q_{F}^{\frac{1}{4}- \theta+ \epsilon} x^{\frac{1}{2}}.
\end{equation}

\begin{rmk}
One has (see Theorem I.5.5,\cite{tenenbaum}), $\underset{d | N_1} \sum \frac{1}{d^{\frac{1}{2}}} \leq e^{2+o(1) \frac{\log^{\frac{1}{2}}N_{1}}{\log_{2} N_{1}}}$. So we would achieve a slightly better bound using this inequality. But for simplicity we are using the bound $N_{1}^{\epsilon}$ here.
\end{rmk}

\subsection{Lower bound} \label{lbd}
From \eqref{naive-Spinor} and \eqref{l-fn}, comparing the Euler factors we have that for $p \nmid  N$
\begin{equation} \label{sum of two}
\lambda_{F}(p) =\lambda_{f}(p) + \lambda_{g}(p) \quad \text{and}
\end{equation}
\[
\lambda_{F}(p^2) = \lambda_{f}(p^2)+ \lambda_{g}(p^2)+ \lambda_{f}(p)\lambda_{g}(p)- \frac{1}{p}.
\]
Hence, for $p \nmid N$, using the Hecke relation
\[\lambda_{f}(p^2) =\lambda_{f}(p)^2 -1 \quad \text{and} \quad \lambda_{g}(p^2) =\lambda_{g}(p)^2 -1, \]
we get from \eqref{sum of two}
\begin{equation} \label{lam}
\lambda_{F}(p)^2 - \lambda_{F}(p^2) =2+\frac{1}{p}+\lambda_{f}(p)\lambda_{g}(p).
\end{equation}
We look for lower bounds for $\lambda_{F}(p)$ from \eqref{lam}, exploiting the nonnegativity of $\lambda_{F}(p^2)$ for $p \leq x^{1/2}$. This leads us to look for a sizeable set of primes on which both $\lambda_{f}(p),\lambda_{g}(p)$ are small. To this end,  we shall first prove the following lemma which will lead us to the required lower bound. See \cite[Lemma 3.1~(iv)]{km} for a result related to this lemma (where a lower bound version has been done).
\begin{lem} \label{mod lambda}
Let $\pi(y, \mathcal{L}) := \# \{ p: p \leq y, p \nmid \mathcal{L} \}$. Let $h \in S_{\kappa} (\mathcal{L})$ be a newform and $\lambda_{h}(n)$ denote its normalised Fourier coefficients. Then for any $y \geq 2$, we have
\[\underset{p \leq y} \sum {}^{\mathcal{L}} | \lambda_{h}(p) | \leq \underset{p \leq y} \sum {}^{\mathcal{L}}(\frac{11}{10} -\frac{57}{1000} \lambda_{h}(p^4)+ \frac{399}{1000} \lambda_{h} (p^2)) .\] 
\end{lem}
\begin{proof}
We know that for $p \nmid \mathcal{L}$, the Ramanujan bound for $|\lambda_{h}(p)|$ is $2$ and
\[\lambda_{h}(p^2) =\lambda_{h}(p)^2 -1 \quad and \quad \lambda_{h}(p^4)=\lambda_{h}(p)^4 -3 \lambda_{h}(p)^2 +1.\]
Before proceeding further, let us first discuss the idea of the proof. If we can find $\delta, \alpha, \beta \in \mathbf{R}$, with $\delta > 0 $  as small as possible, such that
\begin{equation} \label{ineq}
t \leq \delta +\alpha (t^4-3t^2+1) + \beta(t^2-1)
\end{equation}
for all $0 \leq t \leq 2$, then we would have that
\begin{equation} \label{final}
\underset{p \leq y} \sum {}^{\mathcal{L}} | \lambda_{h}(p) | \leq  \underset{p \leq y} \sum {}^{\mathcal{L}} (\delta+ \alpha \lambda_{h}(p^4) +\beta \lambda_{h}(p^2)).
\end{equation}
We put $\beta= \alpha \Upsilon$ and rewrite the polynoimial in the right hand side of \eqref{ineq} as 
\begin{equation}
q(t):= \delta + \alpha (t^4+(-3+\Upsilon) t^2 +1- \Upsilon).
\end{equation}
Let us also define $r(t):= q(t)- t$. We want to find $\alpha, \Upsilon$, such that $r(t)>0$ for all $t \in [0,2]$. Now note that if the derivative of $r(t)$ has no root in $ (0,2)$ and if $r(0),r(2)>0$, then $r(t)>0$ for all $t \in [0,2]$. We have
\[r'(t)= 4 \alpha t^3 +2 \alpha t (-3+ \Upsilon)- 1.\]
For  given $\alpha, \Upsilon$, when $t$ is very close to zero, $r'(t)$ is negative valued. Hence $r'(t)$ has to be negative valued for all $t \in (0,2)$. To ensure this, we want to see, whether there exists $\alpha$, $\Upsilon$, such that $r'(t)$ has a maximum in $t>0$ (note, as a degree $3$ polynomial it can have at most one maximum) and at the point of maximum, $r'(t)$ is negative. We observe that \textit{if $\alpha$ is negative and $\Upsilon < 3$}, then $r'(t)$ has a maximum at $t= (\frac{3- \Upsilon}{6})^{1/2}$. To  ensure that $r'(t)$ is negative for all $t \in (0,2)$ we also want $r'(t)|_{(\frac{3- \Upsilon}{6})^{\frac{1}{2}}} <0$. This yields to the condition
\begin{equation} \label{b}
-8 \alpha < \frac{6^{\frac{3}{2}}}{(3- \Upsilon)^{\frac{3}{2}}}.
\end{equation}
The conditions $q(0) >0$ and $q(2) >2$ are equivalent to
\begin{equation} \label{c}
\delta+(1- \Upsilon) \alpha >0 , \quad \delta+(5+3 \Upsilon) \alpha >2.
\end{equation} 
Using techniques from non-linear programming we get many solutions to this set of simultaneous inequalities \eqref{b}, \eqref{c} together with the condition $\alpha < 0$ and $\Upsilon < 3$. We take the solution
\[ \delta= \frac{11}{10}, \quad \alpha= - \frac{57}{1000}, \quad \Upsilon= -7. \]
Hence we get
\[ \frac{11}{10} -\frac{57}{1000} \lambda_{h} (p^4) +\frac{399}{1000} \lambda_{h}(p^2) \geq | \lambda_{h}(p) | \]
for all $p \nmid \mathcal{L}$. Hence the lemma follows.
\end{proof}

\begin{rmk}
$\frac{11}{10}$ is not the optimal choice for $\delta$. Since the optimal value for $\delta$ improves the lower bound of $S(F,x)$ only up to a constant, we keep $\delta= \frac{11}{10}$.
\end{rmk}

Recall that $Q_{g} \asymp N_2$. Let $L(sym^2 g,s)$ and $L(sym^4 g, s)$ be the symmetric square and symmetric fourth power $L$-functions associated with $g$. It is well known that (see e.g., \cite[chapter 5.1, 5.12]{iwaniec})
\newline
(i) $Q_{sym^2 g} \asymp N_{2}^2$ and $Q_{sym^4 g} \asymp N_{2}^4$ and \newline
(ii) $\lambda_{g}(p^2),\lambda_{g}(p^4)$ are $p$-th coefficients of the Dirichlet series which represent $L(sym^2 g,s)$ and $L(sym^4 g, s)$ respectively.

\begin{prop} \label{propo}
There exists an absolute constant $c_1>0$ such that if $c_1 \log y \geq  (\log Q_{g})^2$, then one has 
\[ \underset{p \leq y} \sum {}^{N_2} |\lambda_{g} (p)| \leq \left( \frac{11}{10}+ O(\frac{1}{\log y}) \right) \pi(y, N_2) .\]
\end{prop}

\begin{proof}
From the holomorphy and non-vanishing at $s=1$ of the symmetric square and the symmetric fourth power $L$-functions (see \cite{kimsha}, \cite[chapter 5.12]{iwaniec}) and the prime number theorem of $L$-functions (see \cite{iwaniec}, pp 110--111), we have for some absolute constant $c_0 >0$ that
\[\underset{p \leq y} \sum {}^{N_2} \lambda_{g}(p^2) \log p = O(N_2 \, y \, \exp (-c_0 \, \sqrt{\log y} ))  \quad \text{and} \]
\[ \underset{p \leq y} \sum {}^{N_2} \lambda_{g}(p^4) \log p = O(N_{2}^2 \, y \, \exp (-c_0 \, \sqrt{\log y} )). \]
So, when $N_2 \ll \exp ( \tfrac{c_0}{4} \sqrt{\log y} )$ using Abel's  summation formula we obtain
\[\underset{p \leq y} \sum {}^{N_2} \lambda_{g}(p^2) =O(\frac{y}{\log^{2} y}) \quad \text{and} \quad \underset{p \leq y} \sum {}^{N_2} \lambda_{g}(p^4) =O(\frac{y}{\log^{2} y}) .\]
The proposition now follows immediatey from \lemref{mod lambda}.
\end{proof}

For any $\gamma>0$, let us define $ V(y, \gamma) := \{p : p \leq y, p \nmid {N_2}, | \lambda_{g}(p) | \leq \gamma \}. $

\begin{cor} \label{cor}
Let the assumptions be as in \propref{propo}. Then at least one of these two following inequalities  holds true:
 \[ \frac{ \# V(y, \frac{19}{20})}{ \pi(y, N_2) } \geq \frac{1}{100} \quad \text{or} \quad \frac{ \#  V(y, \frac{13}{10}) }{\pi(y,N_2)} \geq \frac{51}{100}. \]
\end{cor}

\begin{proof}
We appeal to \propref{propo} and choose $y$ large enough so that the $O(1/\log y)$ term is less than $1/1000$. Suppose none of the the inequalities mentioned above holds. From the negation of the first inequality we get
\begin{equation} \label{one}
\frac{ \# \{p : | \lambda_{g}(p) | > \frac{19}{20}, p \leq y, p \nmid N_2 \}}{ \pi(y, N_2) } > \frac{99}{100}.
\end{equation}
From the negation of second inequality we get
\begin{equation} \label{two}
\frac{ \# \{p : | \lambda_{g}(p) | > \frac{13}{10}, p \leq y ,p \nmid N_2 \}}{  \pi(y, N_2) } > \frac{49}{100}. 
\end{equation}
Hence, combining \eqref{one} and \eqref{two} we get
\[ \underset{p \leq y} \sum {}^{N_2} | \lambda_{g} (p) | > (\frac{49}{100} \cdot \frac{13}{10} + \frac{50}{100} \cdot \frac{19}{20})  \pi(y, N_2) = \frac{1112}{1000} \pi(y, N_2). \]
This is a contradiction with \propref{propo} and hence the corollary follows.   
\end{proof}

We note that
\[ S(F,x)= \underset{n \leq x} \sum {}^N \lambda_{F}(n) \log(\frac{x}{n}) \gg  \underset{n \leq \frac{x}{2}} \sum {}^N \lambda_{F} (n). \]
Thus it is enough to find a lower bound for $\sum_{n \leq x, (n, N)=1} \lambda_{F} (n)$. Recall that $N=\mrm{lcm}(N_1,N_2)$.

\begin{prop} \label{lwbdprop}
Let $c_1$ be the same absolute constant as in \propref{propo} and let $x \geq 0$ be such that $\tfrac{c_1}{2}  \log x \geq (\log Q_{g})^2$. Then we have, under the assumption that $\lambda_F(n) \geq 0$ for all $n \leq x$, that
\begin{align} \label{lwbd}
\underset{n \leq x} \sum {}^N \lambda_{F} (n) \gg \frac{x}{\log^{2} x} .
\end{align} 
\end{prop}

\begin{proof}
We appeal to \corref{cor} with $y= x^{\frac{1}{2}}$ (the assumption that $\tfrac{c_1}{2}  \log x \geq (\log Q_{g})^2$ allows us to choose $y=x^{\frac{1}{2}}$). As the number of distinct prime factors of $N_2$ is at most $\log N_2 / \log 2$, we get
\[ \pi(x^{\frac{1}{2}}, N_2) \geq \pi (x^{\frac{1}{2}})- \frac{\log N_2}{\log 2}  \gg \frac{x^{\frac{1}{2}}}{\log x}.\] 
If $g$ satisfies the first inequality in \corref{cor}, then one has $\# V(x^{\frac{1}{2}}, \frac{19}{20}) \gg \frac{x^{\frac{1}{2}}}{\log x}$. So in this case, for $p \in V({x^{\frac{1}{2}}, \frac{19}{20}})$, from \eqref{lam} we have $\lambda_{F} (p) > 10^{-1/2}$. That gives us
\[\underset{n \leq x} \sum {}^N \lambda_{F} (n) \geq \underset{p_1, p_2 \in V(x^{\frac{1}{2}}, \frac{19}{20})}{\sum_{p_1 \neq p_2}} \lambda_{F}(p_{1} p_{2}) \gg (\underset{p  \in V(x^{\frac{1}{2}}, \frac{19}{20})} \sum 1)^2 - \underset{p  \in V(x^{\frac{1}{2}}, \frac{19}{20})} \sum 1 \gg \frac{x}{\log^2 x}. \]

If $g$ satisfies the second inequality in \corref{cor}, then one has $\# V(x^{\frac{1}{2}}, \frac{13}{10}) \gg x^{\frac{1}{2}}/ \log x$. We split $V(x^{\frac{1}{2}}, \frac{13}{10}) $ into two disjoint sets: \newline 
(I) those $p$ for which $|\lambda_{f}(p)| \geq 14/10$, in which case from \eqref{sum of two} one gets $\lambda_{F}(p) \geq 1/10$; \newline
(II) those $p$ for which $|\lambda_{f}(p)| < 14/10$, in which case from \eqref{lam} one gets $\lambda_{F}(p) > 3 \sqrt{2} /10$. 

Thus combining all the cases above we have,
\[ \underset{n \leq x} \sum {}^N \lambda_{F} (n) \geq \underset{p_1, p_2 \in V(x^{1/2}, \frac{13}{10})}{\sum_{p_{1} \neq p_{2}}} \lambda_{F}(p_{1} p_{2}) \gg (\underset{p  \in V(x^{\frac{1}{2}}, \frac{13}{10})} \sum 1)^2 - \underset{p  \in V(x^{\frac{1}{2}}, \frac{13}{10})} \sum 1 \gg \frac{x}{\log^2 x}. \qedhere  \]
\end{proof}

\subsection{Proof of \thmref{mthm}}
Suppose to the contrary that the first sign change in the sequence $\lambda_F(n)$ occurs after $x$ and that 
\begin{equation} \label{contra}
x \gg_\epsilon Q_F^{1/2-2 \theta + \epsilon}
\end{equation}
for a given $\epsilon >0$. We then look at the upper bound \eqref{upbd} for $S(F,x)$ (replacing $\epsilon$ by $\epsilon/8$ there) and the lower bound from \eqref{lwbd} (note that from our hypothesis in \thmref{mthm} that $(\log Q_{g})^2 \leq c \log Q_f$ for some absolute constant $c$, it follows easily that $\tfrac{c_1}{2} \log x \geq (\log Q_{g})^2$ for some absolute constant $c_1$, so that we can apply \propref{lwbdprop}). This leads to the inequality
\[ \frac{x}{\log^{4} x} \ll_{\epsilon} Q_{F}^{\frac{1}{2}- 2\theta+ \epsilon/4}. \]
As $\epsilon>0$ is arbitrary, applying Lemma $4$ of \cite{choie} we conclude
\begin{equation} \label{last}
x \ll_{\epsilon} Q_{F}^{\frac{1}{2}- 2\theta+ \epsilon/4} (\log Q_{F})^4 \ll Q_{F}^{\frac{1}{2}- 2\theta+ \epsilon/2}. 
\end{equation}
We arrive at a contradiction with \eqref{contra}. This completes the proof of \thmref{mthm}.


\end{document}